\documentclass[11pt,a4paper]{article}
\title{\bf Weak $(1,1)$ Boundedness of Riesz Transforms on Vector Bundles}
\author{Huaiqian Li\footnote{Email: {\color{blue}huaiqianlee@gmail.com}. Supported in part by the National Natural Science Foundation of China (Grant No. 11831014). }\vspace{3mm}\\
{\footnotesize Center for Applied Mathematics, Tianjin University,
 Tianjin 300072, P. R. China}
}
\date{}
\usepackage{amssymb,amsmath,amsfonts,amsthm,color,mathrsfs}

\def\R{\mathbb{R}}

\def\D{\mathbb{D}}

\def\d{\textup{d}}
\def\D{\textup{D}}

\def\cut{\textup{cut}}

\def\End{\textup{End}}
\def\Hess{\textup{Hess}}
\def\Hom{\textup{Hom}}

\def\Ric{\textup{Ric}}
\def\tr{\textup{tr}}

\def\vol{\textup{vol}}
\def\<{\langle}
\def\>{\rangle}
\def\Proof.{\noindent{\bf Proof. }}

\def\newdot{{\kern.8pt\cdot\kern.8pt}}

\newtheorem{theorem}{Theorem}[section]
\newtheorem{lemma}[theorem]{Lemma}

\theoremstyle{definition}

\begin{document}
\allowdisplaybreaks
\maketitle
\makeatletter 
\renewcommand\theequation{\thesection.\arabic{equation}}
\@addtoreset{equation}{section}
\makeatother 

\begin{abstract}
The weak $(1,1)$ boundedness of (local) Riesz transforms corresponding to a large class of Schr\"{o}dinger operators on vector bundles is proved, mainly assuming the generalized volume doubling condition, either Gaussian or sub-Gaussian upper bounds for the heat kernel only in short time, and derivative estimates of Bismut type for the corresponding semigroup.
\end{abstract}


\section{Introduction}\hskip\parindent
Let $M$ be be a complete and non-compact Riemannian manifold, $\vol$ be the Riemannian volume measure, and $\Delta$ be the Laplace--Beltrami operator. Let $(q_t)_{t>0}$ be the heat kernel corresponding to $\Delta$ and $B(x,r)$ be the open ball in $M$ with center $x$ and radius $r>0$. Denote $\vol(x,r)=\vol\big(B(x,r)\big)$. The main theme of the Riesz transform on Riemannian manifolds, denoted by $\nabla(-\Delta)^{-1/2}$, is on the weak $(1,1)$ boundedness, i.e.,
$$\vol\{x\in M:\, |\nabla(-\Delta)^{-1/2} f(x)|\geq\sigma\}\lesssim \frac{1}{\sigma}\int_M|f|\,\d\vol,\quad f\in C_c^\infty(M),\,\sigma>0,$$
and the $L^p$ boundedness, i.e., for some or all $p\in (1,\infty)$,
$$\|\nabla(-\Delta)^{-1/2} f\|_{L^p(M,\vol)}\lesssim \|f\|_{L^p(M,\vol)},\quad f\in C_c^\infty(M).$$
For instance, R. Strichartz asked in \cite{Strichartz1983} that on what non-compact Riemannian manifolds and for which $p\in (1,\infty)$, the Riesz transform $\nabla(-\Delta)^{-1/2}$ is $L^p$ bounded.

In \cite[Theorem 1.1]{CoulhonDuong1999}, under the volume doubling condition, i.e., there exists a constant $C>0$ such that
\begin{eqnarray}\label{doubling}
\vol(x,2r)\leq C \vol(x,r),\quad  x\in M,\, r>0,
\end{eqnarray}
 and the Gaussian upper bound for the heat kernel, i.e., there exist constants $C_1,C_2>0$ such that
\begin{eqnarray}\label{gaussian}
q_t(x,y)\leq\frac{C_1}{\vol(x,\sqrt{t})}\exp\Big\{-C_2\frac{d^2(x,y)}{t}\Big\},\quad x,y\in M,\, t>0,
\end{eqnarray}
the Riesz transform was proved to be weak $(1,1)$ bounded (and hence $L^p$ bounded for all $p\in (1,2]$ by interpolation since the $L^2$ boundedness holds obviously). Let $m>2$. Recently, under the volume doubling condition \eqref{doubling} and the sub-Gaussian upper bound for the heat kernel, i.e., there exist constants $C_1,C_2>0$ such that, for any $x,y\in M$,
 \begin{equation}\label{sub-Gaussian}
q_t(x,y)\leq
 \begin{cases}
 \frac{C_3}{\vol(x,\sqrt{t})}\exp\Big\{-C_4\frac{d^2(x,y)}{t}\Big\},\quad&{ t\in(0,1)},\cr
 \frac{C_3}{\vol(x,t^{1/m})}\exp\Big\{-C_4\Big(\frac{d^m(x,y)}{t}\Big)^{\frac{1}{m-1}}\Big\},\quad&{ t\geq1},
 \end{cases}
 \end{equation}
 the Riesz transform was also proved to be weak $(1,1)$ bounded; see \cite[Theorem 1.2]{CCFR2017} (which also includes results on graphs). Furthermore, on a large class of Riemannian manifolds, under \eqref{doubling} and generalized upper bounds on the heat kernel $q_t$ and on its gradient, the Riesz transform was proved to be weak $(1,1)$ bounded in \cite{LiHQ2018}, where a typical example is given as the direct product Riemannian manifolds such that each element satisfies \eqref{doubling} and the Gaussian \eqref{gaussian}  or the sub-Gaussian \eqref{sub-Gaussian} upper bound for the heat kernel. On proofs of the aforementioned results, estimates on the gradient of the heat kernel $q_t$ play a crucial role. However, similar gradient estimates seem not easy to get for heat kernels on vector bundles since they are just linear operators between fibers. Due to this gap, derivative estimates for semigroups on vector bundles was established, and then applied to prove the weak $(1,1)$ boundedness of the (local) Riesz transform corresponding to a large class of Schr\"{o}dinger operators on vector bundles under the assumption of Gaussian upper bound \eqref{gaussian} for the heat kernel $q_t$ and some kind of volume doubling condition; see \cite[THEOREMS 2.1 and 4.1]{TW2004}.

 We should mention that, assuming non-negative Ricci curvature, weak $(1,1)$ boundedness of Riesz transforms on Riemannian manifolds is proved in \cite{LiJ1991,Chen1992} via gradient heat kernel estimates, while instead of weak $(1,1)$ boundedness, there are many other works on the $L^p$ boundedness (for $1<p<\infty$) of Riesz transforms for differential forms on Riemannian manifolds by different approaches mainly assuming that the Ricci curvature on forms is uniformly bounded from below; among many other works, see e.g. \cite{Ba1987,LiX2009,LiX2010,LiX2013,LiX2014,CMO2015,van2018}, as well as the recent one \cite{BBC2020} (in Section 3 of which the vector bundle case is also studied by a probabilistic approach).

Motivated by the papers \cite{TW2004,CCFR2017,LiHQ2018}, in this work, we consider the  weak $(1,1)$ boundedness of the (local) Riesz transform corresponding to a large class of Schr\"{o}dinger operators on vector bundles. In Section 2, we introduce the framework and recall basic results. In Section 3, we present the main result (Theorem \ref{main-thm}) and its proof.

\section{Preliminaries}\hskip\parindent
Let $E\rightarrow M$ and $F\rightarrow M$ be vector bundles over the same  (not necessarily complete) Riemannian manifold $M$, equipped with metric connections  $\nabla^E$ and $\nabla^F$ respectively. Denote $TM$ and $T^*M$ the tangent and the cotangent bundle of $M$, respectively. We use $\Gamma_{C^\infty}(\bullet)$ (resp. $\Gamma_{C^\infty_b}(\bullet)$ and  $\Gamma_{C^\infty_c}(\bullet)$) to denote the class of smooth (resp.  bounded smooth and  compactly supported smooth) sections of a vector bundle ``$\bullet$''.

Let $\omega\in\Gamma_{C^\infty}(\Hom(T^*M\otimes E, F))$ be a multiplication map, where $\Hom(\bullet,\bullet)$ is the Hom-bundle. Introduce the Dirac type operator from $\Gamma_{C^\infty}(E)$ to $\Gamma_{C^\infty}(F)$ defined by
$$\D_\omega:=\omega\nabla^E,$$
which is a first order differential operator and can be regarded as the composition:
$$\Gamma_{C^\infty}(E)\xrightarrow{\nabla^E}\Gamma_{C^\infty}(T^*M\otimes E)\xrightarrow{\omega}\Gamma_{C^\infty}(F).$$
The Bochner Laplacian $(\nabla^E)^*\nabla^E: \Gamma_{C^\infty}(E)\rightarrow \Gamma_{C^\infty}(E)$ is the second order elliptic differential operator given by the composition:
$$\Gamma_{C^\infty}(E)\xrightarrow{\nabla^E}\Gamma_{C^\infty}(TM\otimes E)\xrightarrow{\nabla^{TM}\otimes1+1\otimes\nabla^E}\Gamma_{C^\infty}(TM\otimes TM\otimes E)\xrightarrow{\tr}\Gamma_{C^\infty}(E),$$
where $\tr$ is the trace operator with respect to the Riemannian metric of $M$ and $\nabla^{TM}$ is the Riemannian connection on $TM$, and the same as the Bochner Laplacian $(\nabla^F)^*\nabla^F: \Gamma_{C^\infty}(F)\rightarrow \Gamma_{C^\infty}(F)$. 

Let $V\in C^2(M)$, $\mathcal{U}_E\in \Gamma_{C^\infty}(\End(E))$ and $\mathcal{U}_F\in\Gamma_{C^\infty}(\End(F))$.  Consider Schr\"{o}dinger type operators
$$L=-(\nabla^E)^*\nabla^E + \nabla^E_{\nabla V}-\mathcal{U}_E$$
on $\Gamma_{C^\infty}(E)$, and
$$T=-(\nabla^F)^*\nabla^F + \nabla^F_{\nabla V}-\mathcal{U}_F$$
on $\Gamma_{C^\infty}(F)$. Note in passing that if both $E$ and $F$ are the same trivial bundle $M\times\R$, then both $L$ and $T$ are just  Schr\"{o}dinger operators on $M$ of the type $\Delta +\nabla V+ U$, where $U:M\rightarrow\R$ is a real potential.

As in \cite{TW2004}, we make the standing assumption that
$$\vartheta=T\D_\omega-\D_\omega L$$
is of zeroth order, i.e., $\vartheta\in\Gamma_{C^\infty}(\Hom(E,F))$, and $\omega$ is compatible with the Riemannian connection, which means that for any $Z\in\Gamma_{C^\infty}(TM)$, $\alpha\in \Gamma_{C^\infty}(E)$ and $v\in TM$, it holds that
$$\nabla^F_v \big(\omega(Z^\flat\otimes\alpha)\big)=\omega\big((\nabla^{TM}_v Z\big)^\flat\otimes\alpha)+\omega(Z^\flat\otimes\nabla^E_v\alpha),$$
where $\flat: TM\rightarrow T^*M$ is the music isomorphism.

Set $\d\mu=e^V\d\vol$, where $\vol$ is the Riemannian volume measure on $M$. Let $(\bullet,\bullet)_E$ (resp. $(\bullet,\bullet)_F$) denote the scalar product and $|\bullet|_E$ (resp. $|\bullet|_F$) the induced norm on fibers of $E$ (resp. $F$).  Let $p\in [1,\infty]$.  Denote  $\Gamma_{L^p_\mu}(E)$ the real Banach space of measurable sections $\alpha: M\rightarrow E$ such that $\|\alpha\|_p<\infty$, where
\begin{eqnarray*}
\|\alpha\|_p:=
\begin{cases}
\big(\int_M|\alpha(x)|_E^p\,\d\mu(x)\big)^{1/p},\quad&{p\in[1,\infty)},\\
\inf\{c\geq0:\, |\alpha|_E\leq c\, \mu\mbox{-a.e.}\},\quad&{p=\infty}.
\end{cases}
\end{eqnarray*}
It turns out that for every $p\in[1,\infty)$, $\Gamma_{L^p_\mu}(E)$ is the closure of $\Gamma_{C_c^\infty}(E)$ with respect to the norm $\|\bullet\|_p$.

We further assume that $\mathcal{U}_E$ is symmetric, i.e., for every $x\in M$,  $\mathcal{U}_E(x)$ is a symmetric linear operator from fiber $E_x$ to itself. It is well known that if $\mathcal{U}_E$ is lower bounded, then $(L,\Gamma_{C^\infty_c}(E))$ is upper bounded in $\Gamma_{L^2_\mu}(E)$ and hence it has a canonical self-adjoint extension, namely, the Friedrich extension in $\Gamma_{L^2_\mu}(E)$, still denoted by $L$. Let $(P_t)_{t\geq0}$ be the semigroup corresponding to $L/2$.

Denote $(P^0_t)_{t\geq0}$ the semigroup corresponding to the Friedrich extension of $((\Delta+\nabla V)/2, C_c^\infty(M))$ in $L^2_\mu(M)$.

Now we recall a derivative estimate of $P_t$, which was established via the generalized Bismut formula under some further assumptions (see \cite[THEOREM 2.1]{TW2004}). Let $\|\cdot\|$ denote the operator norm.

\medskip

\noindent\textbf{Hypothesis} (I). There exist some constants $a_1,a_2,a_3\in\R$ and $C(\vartheta), C(\omega)\geq0$ such that
\begin{itemize}
\item[(I.1)] $a_1|\alpha|_E^2\leq (\mathcal{U}_E\alpha, \alpha)_E\leq a_2|\alpha|_E^2$, for every $\alpha\in E$,
\item[(I.2)] $(\mathcal{U}_F\beta,\beta)_F\geq a_3|\beta|_F^2$, for every $\beta\in F$,
\item[(I.3)] $\|\vartheta\|\leq C(\vartheta)$, $\|\omega\|\leq C(\omega)$.
\end{itemize}

Let $d$ be the Riemannian distance on $M$. For $x\in M$, let $d_x:=d(x,\bullet)$ and denote $\cut(x)$  the cut locus of $x$ in $M$.
\begin{theorem}\label{deri-estimate}
Let $M$ be a complete Riemannian manifold and assume that for every $x\in M$, there exist constants $c>0$ and $\delta\in (0,1)$ such that
$$(\Delta+\nabla V)d_x\leq c(d_x^{-1}+d_x^\delta)$$
outside $\{x\}\cup \cut(x)$. Suppose that \textbf{Hypothesis} (I) holds. Then, for every $\beta\in \Gamma_{C^\infty_b}(E)$ and $x\in M$,
$$|\D_\omega P_t\beta|^2(x)\leq e^{-a_1t}\|\beta\|_\infty\frac{a[C(\omega)+C(\vartheta)\sqrt{t}/2]^2}{1-e^{-at}}P_t^0|\beta|(x),$$
for all $t>0$, where $a:=\max\{a_2-a_3,0\}$ and $a/(1-e^{-at}):=1/t$ if $a=0$.
\end{theorem}

For more details on the framework, refer to \cite{Lee} and \cite{DT2001}. For some concrete examples included in the above context, for instance, differential forms and spinor bundles, refer to \cite[Section 2]{TW2004}.

\section{Riesz transforms on vector bundles}\hskip\parindent
From now on, we assume that $M$ is a complete and non-compact Riemannian manifold, and \textbf{Hypothesis} (I) holds as well as the other assumptions in Section 2. For some suitable constant $\lambda\in\R$, let us define the (local) Riesz transform $R_\lambda$ associated with the operator $L$ by
 $$R_\lambda\alpha=\D_\omega(-L+\lambda)^{-1/2}\alpha,\quad \alpha\in\Gamma_{C^\infty_c}(E).$$
We shall consider the weak $(1,1)$ boundedness of $R_\lambda$, i.e.,
$$\mu\{|R_\lambda\alpha|>\sigma\}\lesssim\frac{\mu(|\alpha|)}{\sigma},\quad\mbox{for any }\sigma>0\mbox{ and any }\alpha\in \Gamma_{C^\infty_c}(E).$$
Here and in the sequel, we use the notation $f\lesssim g$ if there exists some constant $C>0$, which is independent of the main parameters,  such that $f\leq Cg$.

Let $p_t^0$ be the heat kernel corresponding to $P^0_t$ with respect to $\mu$. Denote $V(x,r)=\mu\big(B(x,r)\big)$ for every ball $B(x,r)\subset M$.

\medskip

\noindent\textbf{Hypothesis} (II). $m\geq2$ and $M$ is complete and non-compact satisfying that
\begin{itemize}
\item[(II.1)] for any $x\in M$, there exists a constant $\delta\in (0,1)$ such that
$$(\Delta+\nabla V)d_x\lesssim (d_x^{-1}+d_x^\delta)$$
outside $\{x\}\cup \cut(x)$;

\item[(II.2)] there exist constants $D\geq1$ and $\kappa\in [0,\frac{m}{m-1})$ such that
$$V(x,\tau r)\leq D \tau^D V(x,r)\exp(\tau^\kappa+r^\kappa),\quad    x\in M,\,\tau\geq1,\,r>0;$$

\item[(II.3)] there exists a constant $c>0$ such that
$$p^0_t(x,y)\lesssim\frac{1}{V(x,t^{1/m})}\exp\left\{-c\Big(\frac{d^m(x,y)}{t}\Big)^{\frac{1}{m-1}}\right\},\quad  t\in (0,1],\, x,y\in M.$$
\end{itemize}

Here are some remarks on \textbf{Hypothesis} (II).

(II.1) is just assumption $\textup{B}1$ on page 115 in \cite{TW2004}. It was pointed out that (see lines 1-2 on page 114 of \cite{TW2004}), if there exists some point $o\in M$ such that $\Ric\gtrsim -(1+d_o^{2\delta})$ and $|\nabla V|\lesssim (1+d_o^\delta)$, then (II.1) holds, where $\Ric$ is the Ricci curvature of $M$. We should mention that, (II.2) is not comparable with the local volume doubling condition assumed in \cite[Theorem 1.2]{CoulhonDuong1999} since $\kappa$ is allowed to be bigger than 1, and in particular, when $m=2$, (II.2) is just assumption $\textup{B}2$ on page 115 in \cite{TW2004}. It is well known that, under the completeness assumption on $M$, (II.2) can be derived from $\Ric-\Hess_V\geq0$ with $V$ bounded, where $\Hess_V$ is the Hessian of $V$.

In our Riemannian manifold setting, (II.3) should be understood as that there exists some $t_0\in(0,1)$ such that, for $t\in(0,t_0]$, (II.3) holds with $m=2$, and for $t\in(t_0,1]$, (II.3) holds with $m\geq2$. However, we should mention that the proofs below still work in the situation that (II.3) holds with $m>2$ for $t\in(0,1]$.

For $m=2$, (II.3) is just the Gaussian upper bound, which is equivalent to the assumption $\textup{B}3$ appeared on page 115 in \cite{TW2004} by assuming (II.2). Note that, even in the short-time situation, the sub-Gaussian upper bound (II.3) with $m>2$ and the Gaussian upper bound (II.3) with $m=2$, are not comparable. Moreover, (II.3) with $m>2$ appears naturally as the upper bound of the transition density of a canonical diffusion process on fractal sets with respect to a proper Hausdorff measure. For instance,  the upper and the lower bounds for the transition density of the natural Brownian motion on the Sierpi\'{n}ski gasket in $\R^2$ are comparable with
$$\frac{1}{t^{\upsilon/m}}\exp\left\{-c\Big(\frac{d^m(x,y)}{t}\Big)^{m-1}\right\}, \quad t>0,$$
where $d(x,y)=|x-y|$, $\upsilon=\log3/\log2$ and $m=\log5/\log2>2$; see e.g. \cite{BP1988}.

Let $\Upsilon\in (0,\infty]$. By \cite[Theorem 1.1]{Grig1997}, under the volume doubling condition \eqref{doubling}, the on-diagonal upper bound
\begin{eqnarray}\label{ondiag-upper-bound}
p_t^0(x,x)\lesssim \frac{1}{V(x,\sqrt{t})},\quad  x\in M,\, t\in (0,\Upsilon),
\end{eqnarray}
self-improves to the Gaussian upper bound, i.e., (II.3) with $m=2$ for all $t\in (0,\Upsilon)$. However, this self-improving property for $m>2$, more precisely, from \eqref{ondiag-upper-bound} to (II.3) for all $t\in(0,\Upsilon)$ under the assumption \eqref{doubling}, may not be true. It seems that to find an example to illustrate that the self-improving property does not hold in the sub-Gaussian situation is quite an interesting open problem.

Now we present the main result of this paper.
\begin{theorem}\label{main-thm}
Suppose that \textbf{Hypotheses} (I) and (II) hold and $R_\lambda$ is bounded in $\Gamma_{L^2_\mu}(E)$. Then,
\begin{itemize}
\item[(1)] for $\lambda>-a_1$ (which is specified in (I.1)), $R_\lambda$ is weak $(1,1)$ and bounded in $\Gamma_{L^p_\mu}(E)$ for all $1<p\leq2$;
\item[(2)] if $\vartheta=0$, and either $a_1>0$ or $a_1=\kappa=0$, 
then $R_0$ is weak $(1,1)$ and bounded in $\Gamma_{L^p_\mu}(E)$ for all $1<p\leq2$.
\end{itemize}
\end{theorem}
We should mention that $R_\lambda$ being bounded in $\Gamma_{L^2_\mu}(E)$ is not such a restrictive condition, since in many geometric applications, $\D_\omega$ is just the Dirac operator and $-L$ is just the square of the Dirac operator; for instance, the Hodge Laplacian acting on differential forms, and in that case, $R_\lambda$ is trivially bounded in $\Gamma_{L^2_\mu}(E)$. See also \cite[REMARK 4.5 and COROLLARY 4.6]{TW2004}.

We should point out that, even the result in Theorem \ref{main-thm} in the particular case when $m=2$ and hence $\kappa\in[0,2)$, is not completely covered by the main result \cite[THEOREM 4.1(ii)]{TW2004}, since we do not assume that (II.3) holds for all $t>0$ when $\kappa=0$. Moreover, the method employed below effectively deals with the aforementioned situation and the general case when $m>2$ and $0\leq\kappa<\frac{m}{m-1}$  simultaneously.

In order to prove Theorem \ref{main-thm}, we need some lemmata.  First we present the following one, which shows that a fundamental estimate similar to \cite[Lemma 2.1]{CoulhonDuong1999} also holds under our generalized volume doubling property (II.2).
\begin{lemma}\label{lemma-1}
Suppose that (II.2) holds. Then, for any $r\geq0$ and any $\eta>0$,
\begin{eqnarray*}
&&\int_{M\setminus B(y,r)}\exp\left\{-\eta\Big(\frac{d^m(x,y)}{t}\Big)^{\frac{1}{m-1}}\right\}\,\d\mu(x)\cr
&\lesssim& V(y,t^{1/m})\exp\Big\{-\frac{\eta}{2}\Big(\frac{r^m}{t}\Big)^{\frac{1}{m-1}}\Big\},\quad  y\in M.
\end{eqnarray*}
\end{lemma}
\begin{proof}
Let $r\geq0$ and $\eta>0$. Then
\begin{eqnarray}\label{lemma-eq-1}
&&\int_{M\setminus B(y,r)}\exp\left\{-\eta\Big(\frac{d^m(x,y)}{t}\Big)^{\frac{1}{m-1}}\right\}\,\d\mu(x)\cr
&\lesssim& e^{-\frac{\eta}{2}(\frac{r^m}{t})^{\frac{1}{m-1}}} \int_M\exp\Big\{-\frac{\eta}{2}\Big(\frac{d^m(x,y)}{t} \Big)^{\frac{1}{m-1}}\Big\}\,\d\mu(x).
\end{eqnarray}

For $t\in (0,1]$, applying (II.2), we have
\begin{eqnarray}\label{lemma-eq-2}
&&\int_M\exp\Big\{-\frac{\eta}{2}\Big(\frac{d^m(x,y)}{t} \Big)^{\frac{1}{m-1}}\Big\}\,\d\mu(x)\cr
&\leq& \sum_{i=0}^\infty \int_{B(y,(i+1)t^{1/m})\setminus B(y,it^{1/m}) }\exp\Big\{-\frac{\eta}{2}\Big(\frac{d^m(x,y)}{t} \Big)^{\frac{1}{m-1}}\Big\}\,\d\mu(x)\cr
&\leq& \sum_{i=0}^\infty V(y,(i+1)t^{1/m})\exp\left\{-\frac{\eta}{2}i^{\frac{m}{m-1}}\right\}\cr
&\lesssim& V(y,t^{1/m})\sum_{i=0}^\infty (i+1)^D \exp\left\{(i+1)^\kappa -\frac{\eta}{2}i^{\frac{m}{m-1}}\right\}\cr
&\lesssim& V(y,t^{1/m}),
\end{eqnarray}
where the last line is due to the assumption that $0\leq\kappa<\frac{m}{m-1}$ in (II.2).

For $t>1$, letting  $B_i=B\big(y,(it^{1/m}+1)t^{1/m}\big)$, $i=0,1,2,\cdots$, we have
\begin{eqnarray}\label{lemma-eq-3}
&&\int_M\exp\Big\{-\frac{\eta}{2}\Big(\frac{d^m(x,y)}{t} \Big)^{\frac{1}{m-1}}\Big\}\,\d\mu(x)\cr
&\leq& \mu(B_0) + \sum_{i=1}^\infty \int_{B_i\setminus B_{i-1} }\exp\Big\{-\frac{\eta}{2}\Big(\frac{d^m(x,y)}{t} \Big)^{\frac{1}{m-1}}\Big\}\,\d\mu(x)\cr
&\leq& \mu(B_0) +  \sum_{i=1}^\infty  \mu(B_i) \exp\Big\{-\frac{\eta}{2}[(i-1)t^{1/m}+1]^{\frac{m}{m-1}}\Big\} \cr
&\lesssim& \mu(B_0)\Big[1+ \sum_{i=1}^\infty (it^{1/m}+1)^D \exp\Big\{(it^{1/m}+1)^\kappa + t^{\frac{\kappa}{m}}-\frac{\eta}{2}[(i-1)t^{1/m}+1]^{\frac{m}{m-1}}\Big\} \Big]\cr
&\lesssim& \mu(B_0)\Big[1+ \sum_{i=1}^\infty (it^{1/m}+1)^D \exp\Big\{(2i^\kappa+1) t^{\kappa/m} -\frac{\eta}{2}(i-1)^{\frac{m}{m-1}}t^{\frac{1}{m-1}}\Big\} \Big]\cr
&\lesssim& V(y,t^{1/m}),
\end{eqnarray}
where the last line is again due to that $0\leq\kappa<\frac{m}{m-1}$.

Combing the estimates \eqref{lemma-eq-1}, \eqref{lemma-eq-2} and \eqref{lemma-eq-3}, we complete the proof.
\end{proof}

Next we establish the following result, which means that, under the generalized volume doubling condition (II.2), the short-time upper bound of the heat kernel in (II.3) self-improves to the full-time upper bound with an extra exponential term depending on the time parameter.
\begin{lemma}\label{UE}
Let $m\geq2$. Suppose that (II.2) and (II.3) hold. Then, there exist constants $c_1,c_2>0$ such that
$$p_t^0(x,y)\lesssim\frac{1}{V(y,t^{1/m})}\exp\left\{-c_1\Big(\frac{d^m(x,y)}{t}\Big)^{\frac{1}{m-1}} + c_2t^{\frac{\kappa}{m}}\right\},\quad   x,y\in M,\, t>0.$$
\end{lemma}
\begin{proof}
The main idea of proof is based on the method of induction.

(1) Let $0<t\leq1$. By the symmetry and semigroup property of the heat kernel $p^0_t$, we have
\begin{eqnarray*}
p_{2t}^0(x,y)&=&\int_M p_t^0(x,z)p_t^0(y,z)\,\d\mu(z).
\end{eqnarray*}
Applying the inequality
$$a^q+b^q\geq 2^{1-q}(a+b)^q,\quad \forall\,a,b\geq0, q\geq1, $$
and the triangle inequality $d(x,z)+d(y,z)\geq d(x,y)$, we obtain that, for any $0<\gamma<2c$ with the same $c$ in (II.3),
\begin{eqnarray*}
p_{2t}^0(x,y)
&\leq&\int_M p_t^0(x,z)\exp\left\{\frac{\gamma}{2}\Big(\frac{d^m(x,z)}{t}\Big)^{\frac{1}{m-1}}\right\}p_t^0(y,z)
\exp\left\{\frac{\gamma}{2}\Big(\frac{d^m(y,z)}{t}\Big)^{\frac{1}{m-1}}\right\}\\
&&\times\exp\left\{-c_{\gamma,m}\Big(\frac{d^m(x,y)}{t}\Big)^{\frac{1}{m-1}}\right\}\,\d\mu(z),
\end{eqnarray*}
where $c_{\gamma,m}=\gamma2^{-m/(m-1)}$. Set
$$N_\gamma(x,t)=\int_Mp_t^0(x,z)^2\exp\left\{\gamma\Big(\frac{d^m(x,z)}{t}\Big)^{\frac{1}{m-1}}\right\}\,\d\mu(z).$$
By the Cauchy--Schwarz inequality, we have
\begin{eqnarray*}
p_{2t}^0(x,y)\leq\big(N_\gamma(x,z,t)N_\gamma(y,z,t)\big)^{1/2}\exp\left\{-c_{\gamma,m}\Big(\frac{d^m(x,y)}{t}\Big)^{\frac{1}{m-1}}\right\}.
\end{eqnarray*}
Applying Lemma \ref{lemma-1} with $r=0$, we deduce that
\begin{eqnarray*}
N_\gamma(x,t)&\lesssim&\frac{1}{\big[V(x,t^{1/m})\big]^2}\int_M
\exp\left\{-(2c-\gamma)\Big(\frac{d^m(x,z)}{t}\Big)^{\frac{1}{m-1}}\right\}\,\d\mu(z)\\
&\lesssim&\frac{1}{V(x,t^{1/m})},
\end{eqnarray*}
and
\begin{eqnarray*}
N_\gamma(y,t)\lesssim\frac{1}{V(y,t^{1/m})}.
\end{eqnarray*}
Hence,
\begin{eqnarray*}
p_{2t}^0(x,y)\lesssim\Big(\frac{1}{V(x,t^{1/m})}\frac{1}{V(y,t^{1/m})}\Big)^{1/2}
\exp\left\{-c_{\gamma,m}\Big(\frac{d^m(x,y)}{t}\Big)^{\frac{1}{m-1}}\right\}.
\end{eqnarray*}
By (II.2), $V(y,(2t)^{1/m}) \lesssim V(y,t^{1/m}) e^{t^{\kappa/m}}$, and
\begin{eqnarray*}
V(y,(2t)^{1/m}) &\leq& V(x,d(x,y)+(2t)^{1/m})\\
&\lesssim& V(x,t^{1/m}) \Big(\frac{d(x,y)+(2t)^{1/m}}{t^{1/m}}\Big)^D\\
&&\times\exp\Big\{t^{\kappa/m}+\Big(\frac{d(x,y)+(2t)^{1/m}}{t^{1/m}}\Big)^\kappa\Big\},
\end{eqnarray*}
which imply that
\begin{eqnarray}\label{UE-pf-2}
p_{2t}^0(x,y)&\lesssim&\left( \frac{e^{t^{\kappa/m}}\Big(1+\frac{d(x,y)}{t^{1/m}}\Big)^D \exp\Big[t^{\kappa/m}+\Big(\frac{d(x,y)}{t^{1/m}}\Big)^\kappa\Big]}{\big[V(y,(2t)^{1/m})\big]^2}\right)^{1/2}\cr
&&\times\exp\left\{-c_{\gamma,m}\Big(\frac{d^m(x,y)}{t}\Big)^{\frac{1}{m-1}}\right\}\cr
&\lesssim&\frac{1}{V(y,(2t)^{1/m})}\exp\left\{-\tilde{c}_{\gamma,m}\Big(\frac{d^m(x,y)}{2t}\Big)^{\frac{1}{m-1}}+(2t)^{\kappa/m}\right\},
\end{eqnarray}
for some constant $\tilde{c}_{\gamma,m}>0$. The last inequality of \eqref{UE-pf-2} holds by the assumption that $\kappa<\frac{m}{m-1}$, and the inequality $(1+\xi)^{D/2}e^{-C\xi}\leq C_1e^{-C_2 \xi}$, for some $C_1,C_2>0$ and any $\xi\geq0$, where $C$ is a positive constant.

(2) Similar as step (1) above, for any $t\in(0,1]$ and any $\epsilon\in(0,2c_{\kappa,\gamma,m})$, we have
\begin{eqnarray*}
p_{4t}^0(x,y)\leq\big(N_\epsilon(x,2t)N_\epsilon(y,2t)\big)^{1/2}\exp\left\{-c_{\epsilon,m}\Big(\frac{d^m(x,y)}{2t}\Big)^{\frac{1}{m-1}}\right\},
\end{eqnarray*}
for some constants $c_{\epsilon,m}>0$, where, by applying \eqref{UE-pf-2} and Lemma \ref{lemma-1},
\begin{eqnarray*}
N_\epsilon(x,2t)&\lesssim& \int_M\exp\left\{-(2c_{\kappa,\gamma,m}-\epsilon)\Big(\frac{d^m(x,z)}{2t}\Big)^{\frac{1}{m-1}} + (2t)^{\frac{\kappa}{m}}\right\}\,\d\mu(z)\\
&\lesssim& \frac{e^{(2t)^{\kappa/m}}}{V(x,(2t)^{1/m})},
\end{eqnarray*}
and
$$N_\epsilon(y,2t)\lesssim \frac{e^{(2t)^{\kappa/m}}}{V(y,(2t)^{1/m})}.$$
Hence, by (II.2),
\begin{eqnarray*}
p_{4t}^0(x,y)&\lesssim& \left(\frac{e^{(2t)^{\kappa/m}}}{V(x,(2t)^{1/m})} \frac{e^{(2t)^{\kappa/m}}}{V(y,(2t)^{1/m})} \right)^{1/2}\exp\left\{-c_{\epsilon,m}\Big(\frac{d^m(x,y)}{2t}\Big)^{\frac{1}{m-1}}\right\}\\
&\lesssim& \left( \frac{\Big(2^{1/m}+\frac{d(x,y)}{(2t)^{1/m}}\Big)^D  \exp\Big\{(2t)^{\kappa/m}+\Big(2^{1/m}+\frac{d(x,y)}{(2t)^{1/m}}\Big)^\kappa \Big\} }{V(y,(4t)^{1/m})}  \right)^{1/2}\\
&&\times  \exp\left\{-c_{\epsilon,m}\Big(\frac{d^m(x,y)}{2t}\Big)^{\frac{1}{m-1}}\right\}\\
&\lesssim&  \frac{1}{V(y,(4t)^{1/m})}\exp\left\{-\tilde{c}_{\epsilon,m}\Big(\frac{d^m(x,y)}{4t}\Big)^{\frac{1}{m-1}} + (4t)^{\frac{\kappa}{m}} \right\},
\end{eqnarray*}
for some constant $\tilde{c}_{\epsilon,m}>0$.

(3) Finally, for any $t\in (0,\infty)$, there exists a non-negative integer $N$ such that $t/2^N\in (0,1]$. By the method of induction, similar as the calculation in steps (1) and (2), we complete the proof.
\end{proof}

With these preparations in hand, now we begin to prove Theorem \ref{main-thm}. In fact, the main idea of proof is motivated by \cite{CoulhonDuong1999} and \cite{TW2004}. However, we need some new tricks; see e.g. \eqref{equ-modified} and the proof of \eqref{pf-main-3} below. Let $\mathbf{1}_A$ denote the indicator function of the set $A$. For $c>0$, denote $cB(x,r)=B(x,cr)$ for every ball $B(x,r)$ in $M$.

\begin{proof}[Proof of Theorem \ref{main-thm}]
Let $\alpha\in \Gamma_{C^\infty_c}(E)$, $\sigma>0$ and $N\in\mathbb{N}$. There exists a partition of the support of $\alpha$, denoted by $(E_n)_{n=1}^{N}$, such that each $E_n$ is a bounded domain with diameter no bigger than 1. For each $n\in\{1,2,\cdots,N\}$, we take use of the Calder\'{o}n--Zygmund decomposition (see \cite[LEMMA 4.3]{TW2004}) for $|\alpha|\mathbf{1}_{E_n}$, and then patch them together to obtain that
$$|\alpha|=g+b:=g+\sum_{i=1}^m b_i,$$
for some $m\in\mathbb{N}$, where $g$ and $b_i$ are functions on $M$, and find a sequence of balls $B_i=B(x_i,r_i)$ with $x_i\in M$ and $r_i\in(0,1]$ such that
\begin{itemize}
\item[(a)] $0\leq g(x)\lesssim  \sigma$ for $\mu$-a.e. $x\in M$;
\item[(b)] each $b_i$ is supported in $B_i$ and $\|b_i\|_1\lesssim \sigma\mu(B_i)$;
\item[(c)] $\sum_i\mu(B_i)\lesssim \|\alpha\|_1/ \sigma$;
\item[(d)] every point of $M$ is contained in at most finitely many balls in $\{B_i\}_i$.
\end{itemize}
(See also the proof of \cite[THEOREM 4.1]{TW2004} on pp. 116--117.) It follows immediately from (b) and (c) that
$$\|b\|_1\leq\sum_i \|b_i\|_1\lesssim\|\alpha\|_1,$$and hence $\|g\|_1\lesssim\|\alpha\|_1$.

Since $R_\lambda=\D_m(-L+\lambda)^{-1/2}$, we need to prove that
\begin{eqnarray*}
\mu\{|R_\lambda\alpha|>\sigma\}\lesssim\frac{\|\alpha\|_1}{\sigma},\quad\mbox{for any }\sigma>0\mbox{ and any }\alpha\in \Gamma_{C^\infty_c}(E).
\end{eqnarray*}

For any function $h$ defined on $M$, we let $$\tilde{h}:=h\frac{\alpha}{|\alpha|}\mathbf{1}_{\{|\alpha|>0\}}.$$ Applying the Calder\'{o}n--Zygmund decomposition of $|\alpha|$ at the level $\sigma$, we have that
\begin{eqnarray}\label{pf-main-1-1}
\mu\{|R_\lambda\alpha|>\sigma\}\leq\mu\{|R_\lambda \tilde{g}|>\sigma/2\} + \mu\{|R_\lambda \tilde{b}|>\sigma/2\}.
\end{eqnarray}
Since $R_\lambda$ is assumed to be bounded in $\Gamma_{L^2_{\mu}}(E)$, by (a),  we derive that
\begin{eqnarray}\label{pf-main-1-2}
\mu\{|R_\lambda \tilde{g}|>\sigma/2\}\lesssim \|R_\lambda \tilde{g}\|_2^2/\sigma^2\lesssim \|\tilde{g}\|_2^2/\sigma^2\lesssim \|g\|_1/\sigma\lesssim\|\alpha\|_1/\sigma.
\end{eqnarray}
Hence, it remains to prove that
$$\mu\Big\{|R_\lambda\big(\sum_i\tilde{b}_i\big)|>\sigma/2\Big\} \lesssim \frac{\|\alpha\|_1}{\sigma}.$$

Let $t_i=r_i^m$. We can write
\begin{eqnarray}\label{equ-modified}
R_\lambda \tilde{b}_i=R_\lambda e^{-\lambda t_i}P_{2t_i}\tilde{b}_i+R_\lambda(I-e^{-\lambda t_i}P_{2t_i})\tilde{b}_i.
\end{eqnarray}
Note that the extra term $e^{-\lambda t_i}$ in equation \eqref{equ-modified} is important for achieving our aim. Then, we have
\begin{eqnarray}\label{pf-main-1}
\mu\Big\{|R_\lambda\big(\sum_i \tilde{b}_i\big)|>\sigma/2\Big\}&\leq& \mu\Big\{|R_\lambda\big(\sum_i  e^{-\lambda t_i}P_{2t_i}\tilde{b}_i\big)|>\sigma/4\Big\}\cr
&& + \mu\Big\{|R_\lambda\big(\sum_i (I- e^{-\lambda t_i}P_{2t_i})\tilde{b}_i\big)|>\sigma/4\Big\}.
\end{eqnarray}

We start to estimate the first term on the right hand of \eqref{pf-main-1}. Since $R_\lambda$ is assumed to be bounded in $\Gamma_{L^2_{\mu}}(E)$, by Chebyshev's inequality, we deduce that
\begin{eqnarray*}
\mu\Big\{|R_\lambda\big(\sum_i  e^{-\lambda t_i}P_{2t_i}\tilde{b}_i\big)|>\sigma/4\Big\}&\lesssim& \frac{1}{\sigma^2}\Big\|\sum_i e^{-\lambda t_i}P_{2t_i}\tilde{b}_i\Big\|_2^2\\
&\lesssim& \frac{1}{\sigma^2}\Big\|\sum_i e^{-(a_1+\lambda) t_i} P_{2t_i}^0|b_i|\Big\|_2^2\\
&\leq&\frac{1}{\sigma^2}\Big\|\sum_i P_{2t_i}^0|b_i|\Big\|_2^2,
\end{eqnarray*}
where, in the last but one line, we used the semigroup domination property (see e.g. \cite{HSU1977})
$$|P_s\tilde{b}_i|\leq e^{-a_1 s/2}P_s^0|b_i|,\quad\mbox{for any }s>0,$$
and in the last line, we used the assumption that $a_1+\lambda>0$.
By duality,
\begin{eqnarray*}
\Big\|\sum_i P_{2t_i}^0|b_i|\Big\|_2&=& \sup_{\|f\|_2=1}\Big|\int_M\sum_i \big(P_{2t_i}^0|b_i|\big)f\,\d\mu\Big|\\
&=&\sup_{\|f\|_2=1}\Big|\sum_i \int_M|b_i| P_{2t_i}^0 f\,\d\mu\Big|\\
&\leq&\sup_{\|f\|_2=1}\sum_i \|b_i\|_1\Big(\sup_{B_i}|P_{2t_i}^0 f|\Big).
\end{eqnarray*}
We claim that, for every $f\in L^2(M,\mu)$,
\begin{eqnarray}\label{pf-main-2}
\sup_{B_i}|P_{2t_i}^0f|\lesssim\inf_{B_i}\mathcal{M}(f),
\end{eqnarray}
where $\mathcal{M}$ is the Hardy--Littlewood maximal operator defined as
$$\mathcal{M}(f)(x)=\sup_{r>0}\frac{1}{V(x,r)}\int_{B(x,r)} |f(y)|\,\d\mu(y).$$
Let $y\in B_i$. Set $G_1=2^{2/m}B_i$ and $G_j=2^{(j+1)/m}B_i\setminus 2^{j/m}B_i$ when $j=2,3,\cdots$. Applying Lemma \ref{UE}, we derive that, for any $z\in G_j$,
\begin{eqnarray*}
 p_{2t_i}^0(y,z)\lesssim \frac{1}{V(y,(2t_i)^{1/m})}e^{-c_1 2^{j-1}+c_2(2t_i)^{\kappa/m}}\lesssim \frac{1}{V(y,(2t_i)^{1/m})}e^{-c_1 2^{j-1}},
\end{eqnarray*}
since $t_i\in (0,1]$, where $c_1,c_2>0$ are constants from Lemma \ref{UE}. Hence, for every $f\in L^2(M,\mu)$, by (II.2),
\begin{eqnarray*}
|P_{2t_i}^0f(y)|&=&\bigg|\int_M p_{2t_i}^0(y,z)f(z)\,\d\mu(z)\bigg| \leq \sum_{j=1}^\infty\int_{G_j}p_{2t_i}^0(y,z)|f(z)|\,\d\mu(z)\\
&\lesssim&\sum_{j=1}^\infty \frac{\mu(2^{j+1}B_i)}{V(y,(2t_i)^{1/m})}\frac{e^{-c_12^{j-1}}}{\mu(2^{j+1}B_i)}\int_{2^{j+1}B_i}|f(z)|\,\d\mu(z)\\
&\lesssim& \sum_{j=1}^\infty  2^{(j/m+1)D}e^{2^{(j/m+1)\kappa}}e^{-c_12^{j-1}} \Big(\inf_{B_i}\mathcal{M}(f)\Big)\\
&\lesssim& \inf_{B_i}\mathcal{M}(f),
\end{eqnarray*}
which implies \eqref{pf-main-2}. Thus, if $\|f\|_2=1$, then
\begin{eqnarray*}
&&\sum_i \|b_i\|_1\Big(\sup_{B_i}|P_{2t_i}^0 f|\Big)\lesssim\sum_i \sigma \mu(B_i)\inf_{B_i}\mathcal{M}(f)\\
&\leq& \sigma \sum_i \int_{B_i} \mathcal{M}(f)\,\d\mu\leq \sigma \big[\mu\big(\cup_i B_i\big)\big]^{1/2}\|\mathcal{M}(f)\|_2\\
&\lesssim&\big(\sigma\|\alpha\|_1\big)^{1/2},
\end{eqnarray*}
where we used (b) and (c) above and the fact that $\mathcal{M}$ is bounded in $L^2(M,\mu)$. Thus,
\begin{eqnarray}\label{pf-main-1-4}
\mu\Big\{|R_\lambda\big(\sum_i  e^{-\lambda t_i}P_{2t_i}\tilde{b}_i\big)|>\sigma/4\Big\}\lesssim \frac{\|\alpha\|_1}{\sigma}.
\end{eqnarray}

It remains to estimate the second term on the right hand side of \eqref{pf-main-1}. Obviously,
\begin{eqnarray*}
&&\mu\Big\{|R_\lambda\big(\sum_i (I- e^{-\lambda t_i}P_{2t_i})\tilde{b}_i\big)|>\sigma/4\Big\}\cr
&\leq&\sum_i\mu(2B_i) + \mu\Big(\big(M\setminus \cup_i2B_i)\cap  \Big\{ |R_\lambda\big(\sum_i(I-e^{-\lambda t_i}P_{2t_i})\tilde{b}_i\big)|>\sigma/4\Big\} \Big)\cr
&\leq&\sum_i\mu(2B_i) + \frac{4}{\sigma} \sum_i \int_{M\setminus 2B_i}  |R_\lambda (I-e^{-\lambda t_i}P_{2t_i})\tilde{b}_i |\,\d\mu,
\end{eqnarray*}
where, by (II.2) and (c) above,
$$\sum_i\mu(2B_i)\lesssim \sum_i\mu(B_i)e^{r_i^\kappa}\lesssim\sigma^{-1}\|\alpha\|_1,$$
since $r_i\in(0,1]$. Hence, in view of (b) and (c),  it is sufficient to prove that, for each $i$,
\begin{eqnarray}\label{pf-main-3}
 \int_{M\setminus 2B_i} |R_\lambda (I-e^{-\lambda t_i}P_{2t_i})\tilde{b}_i |\,\d\mu\lesssim\|b_i\|_1.
\end{eqnarray}

Since
$$(-L+\lambda)^{-1/2}=\frac{1}{\sqrt{\pi}}\int_0^\infty e^{-\lambda s}P_{2s}\frac{\d s}{\sqrt{s}},$$
where the integral is understood in Bochner's sense, we obtain that
\begin{eqnarray*}
&&(-L+\lambda)^{-1/2}(I-e^{-\lambda t_i}P_{2t_i})\cr
&=&\frac{1}{\sqrt{\pi}}\int_0^\infty e^{-\lambda s}P_{2s}\frac{\d s}{\sqrt{s}}-
\frac{1}{\sqrt{\pi}}\int_0^\infty e^{-\lambda (s+t_i)}P_{2(s+t_i)}\frac{\d s}{\sqrt{s}}\\
&=&\frac{1}{\sqrt{\pi}}\int_0^\infty\Big(\frac{1}{\sqrt{s}}-\frac{\mathbf{1}_{\{s>t_i\}}}{\sqrt{s-t_i}}\Big)e^{-\lambda s}P_{2s}\,\d s.
\end{eqnarray*}
 Hence, applying Theorem \ref{deri-estimate}, we immediately deduce that
\begin{eqnarray*}
&&|R_\lambda (I-e^{-\lambda t_i}P_{2t_i})\tilde{b}_i | \cr
&\leq& \frac{1}{\sqrt{\pi}}\int_0^\infty\Big|\frac{1}{\sqrt{s}}-\frac{\mathbf{1}_{\{s>t_i\}}}{\sqrt{s-t_i}}\Big|e^{-\lambda s}|\D_\omega P_{2s}\tilde{b}_i|\,\d s\cr
&\leq& \frac{1}{\sqrt{\pi}}\int_0^\infty\Big|\frac{1}{\sqrt{s}}-\frac{\mathbf{1}_{\{s>t_i\}}}{\sqrt{s-t_i}}\Big|e^{-\lambda s}\big( C(s) e^{-a_1s}  \|P_s \tilde{b}_i\|_\infty P^0_s|P_s\tilde{b}_i|\big)^{1/2}\,\d s,
\end{eqnarray*}
where $$C(s):=\frac{a[C(\omega)+C(\vartheta)\sqrt{s}/2]^2}{1-e^{-as}}$$
with $a=\max\{a_2-a_3,0\}$. By the semigroup domination property again,
$$P^0_s|P_s\tilde{b}_i|\leq e^{-a_1s/2}P_{2s}^0|b_i|,$$
and by Lemma \ref{UE}, we have that
\begin{eqnarray*}
|P_s\tilde{b}_i(x)|&\leq& e^{-a_1s/2}\int_{B_i} p_s^0(x,y)|b_i|(y)\,\d\mu(y)\cr
&\lesssim&e^{-a_1s/2}\|b_i\|_1\sup_{y\in B_i}\frac{e^{c_2s^{\kappa/m}}}{V(y,s^{1/m})}.
\end{eqnarray*}
Then
\begin{eqnarray}\label{pf-main-4}
&&|R_\lambda (I-e^{-\lambda t_i}P_{2t_i})\tilde{b}_i | \cr
&\lesssim& \int_0^\infty\Big|\frac{1}{\sqrt{s}}-\frac{\mathbf{1}_{\{s>t_i\}}}{\sqrt{s-t_i}}\Big|e^{-(\lambda +a_1)s}\Big( C(s)\|b_i\|_1 P^0_{2s}|b_i| \sup_{y\in B_i}\frac{e^{c_2s^{\kappa/m}}}{V(y,s^{1/m})}\Big)^{1/2}\,\d s.
\end{eqnarray}

Let $\eta=c_1/2$. By (b) above and the Cauchy--Scharwz inequality, we deduce that
\begin{eqnarray*}
&&\int_{M\setminus 2B_i}\big(P_{2s}^0|b_i|(x)\big)^{1/2}\,\d\mu(x)=\int_{M\setminus 2B_i}\Big(\int_{B_i}p_{2s}^0(x,y)|b_i(y)|\,\d\mu(y)\Big)^{1/2}\,\d\mu(x)\cr
&\leq&\int_{M\setminus 2B_i}\Big(\sup_{y\in B_i}e^{-\eta\big(\frac{d^m(x,y)}{2s}\big)^{\frac{1}{m-1}}}\Big)^{1/2}\Big(\int_{B_i}e^{\eta\big(\frac{d^m(x,y)}{2s}\big)^{\frac{1}{m-1}}}
p_{2s}^0(x,y)|b_i(y)|\,\d\mu(y)\Big)^{1/2}\,\d\mu(x)\cr
&\leq&\Big(\int_{M\setminus 2B_i}\int_{B_i}e^{\eta\big(\frac{d^m(x,y)}{2s}\big)^{\frac{1}{m-1}}}
p_{2s}^0(x,y)|b_i(y)|\,\d\mu(y)\,\d\mu(x)\Big)^{1/2}\cr
&&\times\Big(\sup_{y\in B_i}\int_{M\setminus 2B_i}e^{-\eta\big(\frac{d^m(x,y)}{2s}\big)^{\frac{1}{m-1}}}\,\d\mu(x)\Big)^{1/2}\cr
&=:& J_1\times J_2.
\end{eqnarray*}
Lemma \ref{lemma-1} implies that
$$J_2\lesssim e^{-\frac{\eta}{4}\big(\frac{(2r_i)^m}{s}\big)^{\frac{1}{m-1}}}\sup_{y\in B_i}\big[ V\big(y,(2s)^{1/m}\big)\big]^{1/2}.$$
Lemma \ref{lemma-1} and Lemma \ref{UE} imply that
\begin{eqnarray*}
J_1&\lesssim& \Big(\int_{M\setminus 2B_i} \frac{e^{c_2(2s)^{\kappa/m}}|b_i|(y)}{V(y,(2s)^{1/m})}  \int_{B_i} e^{-(c_1-\eta)\big(\frac{d^m(x,y)}{2s}\big)^{\frac{1}{m-1}}} \,\d\mu(x)\,\d\mu(y)       \Big)^{1/2}\cr
&\lesssim&\big(e^{c_2(2s)^{\kappa/m}}\|b_i\|_1\big)^{1/2}.
\end{eqnarray*}
Hence,
\begin{eqnarray}\label{pf-main-5}
&&\int_{M\setminus 2B_i}\big(P_{2s}^0|b_i|(x)\big)^{1/2}\,\d\mu(x)\cr
&\lesssim&\exp\Big[-\frac{c_1}{8}\Big(\frac{(2r_i)^m}{2s}\Big)^{\frac{1}{m-1}}+ \frac{c_2}{2}(2s)^{\frac{\kappa}{m}}\Big] \big[\|b_i\|_1 \sup_{y\in B_i} V\big(y,(2s)^{1/m}\big)\big]^{1/2}.
\end{eqnarray}

Combining \eqref{pf-main-3}, \eqref{pf-main-4} and \eqref{pf-main-5}, we arrive at
\begin{eqnarray*}
&&\int_{M\setminus 2B_i} |R_\lambda (I-e^{-\lambda t_i}P_{2t_i})\tilde{b}_i |\,\d\mu\cr
&\lesssim&
\|b_i\|_1\int_0^\infty \Big|\frac{1}{\sqrt{s}}-\frac{\mathbf{1}_{\{s>t_i\}}}{\sqrt{s-t_i}}\Big| h(s)
\Big(\sup_{x,y\in B_i}\frac{V(x,(2s)^{1/m})}{V(y,s^{1/m})}\Big)^{1/2},
\end{eqnarray*}
where
$$h(s)=\sqrt{C(s)}\exp\Big[2c_2 s^{\frac{\kappa}{m}} -\frac{c_1}{4}\Big(\frac{t_i}{s}\Big)^{\frac{1}{m-1}} - (a_1+\lambda)s\Big].$$
By (II.2), for any $x,y\in B_i=B(x_i,t_i^{1/m})$,
\begin{eqnarray*}
V\big(x,(2s)^{1/m}\big)&\leq& V\big(y,2t^{1/m}_i+(2s)^{1/m}\big)\\
&\lesssim&\Big(\frac{2t^{1/m}_i+(2s)^{1/m}}{s^{1/m}}\Big)^D V\big(y,s^{1/m}\big) \exp\Big[\Big(\frac{2t^{1/m}_i+(2s)^{1/m}}{s^{1/m}}\Big)^\kappa + s^{\kappa/m}\Big]\\
&\leq& C_{\varepsilon,D}V\big(y,s^{1/m}\big)\exp\Big[\varepsilon\Big(\frac{t_i}{s}\Big)^{\frac{1}{m-1}} + s^{\kappa/m}\Big],
\end{eqnarray*}
for any  $\varepsilon>0$, where $C_{\varepsilon,D}$ is a positive constant. Hence, for any  $\varepsilon>0$,
\begin{eqnarray*}
&&\int_{M\setminus 2B_i} |R_\lambda (I-e^{-\lambda t_i}P_{2t_i})\tilde{b}_i |\,\d\mu\cr
&\lesssim&
\|b_i\|_1\int_0^\infty \Big|\frac{1}{\sqrt{s}}-\frac{\mathbf{1}_{\{s>t_i\}}}{\sqrt{s-t_i}}\Big| k(s)\,\d s\cr
&=&\|b_i\|_1 \Big(\int_0^{t_i}\frac{k(s)}{\sqrt{s}}\,\d s + \int_{t_i}^\infty \Big|\frac{1}{\sqrt{s}}-\frac{1}{\sqrt{s-t_i}}\Big| k(s)\,\d s\Big),
\end{eqnarray*}
where
$$k(s):=\sqrt{C(s)}\exp\Big[-\frac{(c_1-2\varepsilon)}{4}\Big(\frac{t_i}{s}\Big)^{\frac{1}{m-1}} - (a_1+\lambda)s + (1+2c_2) s^{\frac{\kappa}{m}}\Big].$$
Let
$$K_1:=\int_0^{t_i}\frac{k(s)}{\sqrt{s}}\,\d s,\quad K_2:=\int_{t_i}^\infty \Big|\frac{1}{\sqrt{s}}-\frac{1}{\sqrt{s-t_i}}\Big| k(s)\,\d s.$$
For every $\varepsilon\in (0, c_1/2)$, since $a_1+\lambda>0$ and $0\leq\frac{\kappa}{m}<\frac{1}{m-1}\leq1$, it is straightforward to check that
$$K_1\lesssim \int_0^{t_i} \frac{1}{s}\exp\Big[-\frac{(c_1-2\varepsilon)}{4}\Big(\frac{t_i}{s}\Big)^{\frac{1}{m-1}}\Big]\,\d s<\infty,$$
and
$$K_2\lesssim \int_{t_i}^\infty \Big|\frac{1}{\sqrt{s}}-\frac{1}{\sqrt{s-t_i}}\Big| \frac{1+\sqrt{s}}{\sqrt{s}}\exp\Big[- (a_1+\lambda)s + (1+2c_2) s^{\frac{\kappa}{m}}\Big]\,\d s<\infty,$$
where to estimate $C(s)$ we used the elementary inequality that $\frac{a}{1-e^{-as}}\leq \frac{1}{s}$, $s>0$, since $a\geq0$.

Thus, \eqref{pf-main-3} is proved, and hence we get
\begin{eqnarray}\label{pf-main-1-5}
\mu\Big\{|R_\lambda\big(\sum_i (I- e^{-\lambda t_i}P_{2t_i})\tilde{b}_i\big)|>\sigma/4\Big\} \lesssim \sigma^{-1}\|\alpha\|_1.
\end{eqnarray}

Combining \eqref{pf-main-1-1}, \eqref{pf-main-1-2}, \eqref{pf-main-1}, \eqref{pf-main-1-4} and \eqref{pf-main-1-5}, we complete the proof of (1).

\medskip

Now let $\vartheta=0$. We need to prove $R_0=\D_\omega(-L)^{-1/2}$ is weak $(1,1)$ bounded. From the above argument, the only difference lies in the
estimation of $K_2$. If either $a_1>0$ or $a_1=\kappa=0$, then
\begin{eqnarray*}
K_2&\lesssim& \int_{t_i}^\infty \Big|\frac{1}{\sqrt{s}}-\frac{1}{\sqrt{s-t_i}}\Big| \frac{1+\sqrt{s}}{\sqrt{s}}\,\d s\\
&=& \int_0^\infty\Big|\frac{1}{\sqrt{u}}-  \frac{1}{\sqrt{u+1}}  \Big| \Big( \frac{1}{\sqrt{u+1}} +\sqrt{t_i}\Big)\,\d s
<\infty.
\end{eqnarray*}
Thus, (2) is proved.

Therefore, we complete the proof of Theorem \ref{main-thm}.
\end{proof}

Finally, we remark that some further extensions of our main results as in \cite{LiHQ2018} are possible. However, the idea is the same, so we leave these for interested readers.

\end{document}